\documentclass[10pt]{amsart}
\usepackage{amssymb}
\usepackage{epsfig}
\usepackage{url}
\usepackage{setspace}
\usepackage{color}
\theoremstyle{plain}

\newtheorem{thm}{Theorem}[section]
\newtheorem{cor}[thm]{Corollary}
\newtheorem{lem}[thm]{Lemma}

\newtheorem{rem}[thm]{Remark}
\newtheorem{ques}[thm]{Question}
\newtheorem{conj}[thm]{Conjecture}
\newtheorem{prob}[thm]{Problem}

\def\bbb{\mathbb}
\def\op{\operatorname}
\renewcommand{\phi}{\varphi}

\newcommand{\N}{\bbb{N}}
\newcommand{\Z}{\bbb{Z}}
\newcommand{\Q}{\bbb{Q}}

\makeatletter
\let\@@pmod\pmod
\DeclareRobustCommand{\pmod}{\@ifstar\@pmods\@@pmod}
\def\@pmods#1{\mkern4mu({\operator@font mod}\mkern 6mu#1)}
\makeatother

\begin{document}
\title[On the Diophantine equation $\sigma_{2}(\overline{X}_{n})=\sigma_{n}(\overline{X}_{n})$]{On the Diophantine equation $\sigma_{2}(\overline{X}_{n})=\sigma_{n}(\overline{X}_{n})$}
\author{Piotr Miska}
\author{Maciej Ulas}

\thanks{Research of the authors was supported by a grant of the National Science Centre (NCN), Poland, no. UMO-2019/34/E/ST1/00094}

\begin{abstract}
In this note we investigate the set $S(n)$ of positive integer solutions of the title Diophantine equation. In particular, for a given $n$ we prove boundedness of the number of solutions, give precise upper bound on the common value of $\sigma_{2}(\overline{X}_{n})$ and $\sigma_{n}(\overline{X}_{n})$ together with the biggest value of the variable $x_{n}$ appearing in the solution. Moreover, we enumerate all solutions for $n\leq 16$ and discuss the set of values of $x_{n}/x_{n-1}$ over elements of $S(n)$.
\end{abstract}

\maketitle

\begin{center}
    Dedicated to the memory of professor Andrzej Schinzel.
\end{center}

\section{Introduction}\label{sec1}

Let $\N$ be the set of non-negative integers, $\N_{+}$ the set of positive integers and for given $k\in\N$ we define $\N_{\geq k}$ as the set of integers $\geq k$. Moreover, for a given set $A$, by $|A|$ we denote the number of elements of the set $A$.

Let $n\in\N_{\geq 3}$ and for given $k\leq n$ consider the $k$-th symmetric polynomial
$$
\sigma_{k}(x_{1}, \ldots, x_{n})=\sum_{i_{1}<i_{2}<\ldots<i_{k}}x_{i_{1}}\cdots x_{i_{k}}.
$$
In the sequel, for simplicity of notation, we will write $\overline{1}_{k}$ instead of $(1,\ldots, 1)$, where the number of 1's is equal to $k$.

The question whether the sum of elements of a given finite set can be equal to the product of these elements is a classical one. Equivalently, we ask for which $n\in\N_{\geq 2}$ the  Diophantine equation
$$
\sigma_{1}(\overline{X}_{n})=\sigma_{n}(\overline{X}_{n})
$$
has a solution in positive integers $x_{1},\ldots, x_{n}$. This question was investigated by many authors. For each $n\geq 3$ the equation has a solution $(\overline{1}_{n-2}, 2, n)$. In particular, $N(n)\geq 1$, where $N(n)$ is the number of solutions. Schinzel showed that there is no other solution for $n=6$ or $n=24$. He also investigated the existence of rational solutions in the case of $n=3$ and proved that for each $m$ there are at least $m$ triples of integers with the same sum and the same product \cite{Sch}. Misiurewicz has shown that $n=2, 3, 4, 6, 24, 114, 174, 444$ are the only values of $n<10^3$ for which $N(n)=1$ \cite{Mi} (recall that the value 114 was misprinted as 144 in \cite{Mi}). This was later extended by Brown and by Singmaster, Bennett and Dunn to $n\leq 1444000$. All these results were improved by Weingartner. He proved that $N(n)>1$ for $444<n<10^{11}$ \cite{Wei}. This was possible due to connection with Sophie-Germain primes \cite{Br} (see also \cite{Nyb}). The question whether there is a value $n>444$ such that $N(n)=1$ is still open. A nice exposition of the basic results concerning this problem can be found in \cite{Eck}. For more on the history of this problem and related investigations see section D24 in \cite{Guy}.

Motivated by the findings concerning the equation $\sigma_{1}(\overline{X}_{n})=\sigma_{n}(\overline{X}_{n})$ it is natural to ask a question about the existence of positive integer solutions of the Diophantine equation
\begin{equation}\label{maineq}
\sigma_{2}(\overline{X}_{n})=\sigma_{n}(\overline{X}_{n}).
\end{equation}
Let us note that (\ref{maineq}) is equivalent with the Diophantine equation
$$
\sigma_{n-2}\left(\frac{1}{x_{1}},\ldots,\frac{1}{x_{n}}\right)=1
$$
which need to be solved in positive integers and its solutions give quite special representations of 1 in terms of Egyptian fractions. By a solution of (\ref{maineq}) we mean a sequence $\overline{X}_{n}=(x_{1},\ldots, x_{n})$ satisfying the condition $x_{i}\leq x_{i+1}$ for $i=1,\ldots, n-1$.

In this paper we are interested in the structure of the set
$$
S(n)=\{\overline{X}_{n}\in\N_{+}^{n}:\;\overline{X}_{n}\;\mbox{is a solution of}\;(\ref{maineq})\}
$$
and investigate its various properties. Next, for $i\in\{0,\ldots, n\}$, let us put
$$
S_{i}(n)=\{\overline{X}_{n}\in S(n):\;x_{1}=\ldots=x_{n-i}=1,\;2\leq x_{n-i+1}\leq \ldots \leq x_{n}\},
$$
i.e., $S_{i}(n)\subset S(n)$ contain solutions of (\ref{maineq}) with exactly $i$ terms different than 1. We clearly have the (disjoint) decomposition
$$
S(n)=\bigcup_{i=0}^{n}S_{i}(n).
$$

Let us describe the content of the paper in some details. In Section \ref{sec2} finiteness of the set $S(n)$ is proved together with the maximum value of $x_{n}$ which can appear in the solution $\overline{X}_{n}$ of (\ref{maineq}) (Theorem \ref{mainbound}). We also present a bound for $x_{n-2}$ (Corollary \ref{xneqxn}). Using the obtained results we compute all the solutions of (\ref{maineq}) for $n\leq 16$.

In Section \ref{sec3} we investigate the set $S_{3}(n)$ in details. Our findings allow us to prove that
$$
|S(n)|\geq |S_{3}(n)|\geq \frac{1}{2}\tau\left(\frac{1}{2}(n-2)(3n-1)\right),
$$
where $\tau(m)$ is the number of divisors of a positive integer $m$. Moreover, we investigate the behaviour of $x_{n}/x_{n-1}$, where $x_{n-1}, x_{n}$ are components of $\overline{X}_{n}\in S(n)$. In particular, we prove that the set of rational values $x_{n}/x_{n-1}$, where $\overline{X}_{n}\in S(n)$ and $n\in\N_{\geq 3}$, is dense in the set $[1,+\infty)$.

Finally, in the last section  we state several questions and conjectures which, we believe, will motivate further investigations.

\section{Boundedness of $x_{n}$ and enumeration of $S(n)$ for $n\leq 16$}\label{sec2}

Before we start, we mention two basic identities involving symmetric polynomials. More precisely, we have
\begin{equation}\label{sigmaid1}
\sigma_{1}(\overline{X}_{k_{1}},\overline{Y}_{k_{2}})=\sigma_{1}(\overline{X}_{k_{1}})+\sigma_{1}(\overline{Y}_{k_{2}})
\end{equation}
and
\begin{equation}\label{sigmaid2}
\sigma_{2}(\overline{X}_{k_{1}},\overline{Y}_{k_{2}})=\sigma_{1}(\overline{X}_{k_{1}})\sigma_{1}(\overline{Y}_{k_{2}})+\sigma_{2}(\overline{X}_{k_{1}})+\sigma_{2}(\overline{Y}_{k_{2}}),
\end{equation}
where $\overline{X}_{k_{1}}, \overline{Y}_{k_{2}}$ are independent sets of variables. In the sequel we will use these identities several times.

We start with a simple bound for $x_{1}\cdot\ldots\cdot x_{n-2}$.

\begin{lem}\label{upperboundforxn-2}
If $n\geq 3$ and $\overline{X}_{n}\in S(n)$ then $x_{1}\cdot\ldots\cdot x_{n-2}\leq \binom{n}{2}$. In particular, $x_{n-2}\leq \binom{n}{2}$.
\end{lem}
\begin{proof}
If $\overline{X}_{n}\in S(n)$ with $x_{1}\leq\ldots\leq x_{n}$ then
$$
\sigma_{n}(\overline{X}_{n})=\sigma_{2}(\overline{X}_{n})\leq \binom{n}{2}x_{n-1}x_{n},
$$
and dividing by $x_{n-1}x_{n}$ we get the inequality from the statement.
\end{proof}

On the other hand, we have a lower bound $x_{1}\cdot\ldots\cdot x_{n-2}\geq 2$ as the following holds.

\begin{lem}\label{lowerboundforxn-2}
For $n\geq 3$ we have $S_{0}(n)=S_{1}(n)=S_{2}(n)=\emptyset$.
\end{lem}

\begin{proof}
Suppose that $\overline{X}_n\in S_i(n)$ for some $i\leq 2$. Then
$$\sigma_2(\overline{X}_n)\geq x_{n-1}x_n+x_{n-1}+x_n> x_{n-1}x_n= \sigma_n(\overline{X}_n),$$
which is a contradiction.
\end{proof}

Lemma \ref{lowerboundforxn-2} states that if $n\in\N_{\geq 3}$ and $S_i(n)\neq\emptyset$, then $i\geq 3$. However, if $S_i(n)\neq\emptyset$, then $i$ cannot be too big comparing to $n$. This is due to the lemma below.

\begin{lem}\label{inumber}
Let $n\in\N_{\geq 3}$. If $S_i(n)\neq\emptyset$, then $i\leq 2+\log_2\binom{n}{2}$.
\end{lem}

\begin{proof}
Let $\overline{X}_{n}\in S_i(n)$. Then
$$2^{i-2}\leq x_{n-i+1}\cdot\ldots\cdot x_{n-2}\leq\binom{n}{2},$$
where the last inequality follows from Lemma \ref{upperboundforxn-2}. After taking the logarithm with base $2$ from the left hand side and right hand side of the above inequality we get
$$i-2\leq\log_2\binom{n}{2}.$$
The lemma is proved.
\end{proof}

As an immediate consequence we get the following.
\begin{cor}\label{ones}
Let $n\in\N_{\geq 3}$ and $\overline{X}_{n}\in S(n)$. If $n\geq 6$ then $x_{1}=1$. If $n\geq 8$ then $x_{1}=x_{2}=1$.
\end{cor}

Writing $\overline{X}_{n-2}$ we mean $(x_1,\ldots,x_{n-2})$. The next result shows how to find $x_{n-1},x_n\in\N$ such that $\overline{X}_{n}\in S(n)$, where $\overline{X}_{n-2}$ is fixed.

\begin{thm}\label{sols}
Let $n\in\N_{\geq 3}$ and $\overline{X}_{n}\in S(n)$. Then
\begin{align*}
x_{n-1}&=\frac{\sigma_1(\overline{X}_{n-2})+d_{1}}{\sigma_{n-2}(\overline{X}_{n-2})-1},\quad x_n=\frac{\sigma_1(\overline{X}_{n-2})+d_{2}}{\sigma_{n-2}(\overline{X}_{n-2})-1},
\end{align*}
where $d_1,d_2\in\N$ are such that
$$d_{1}d_{2}=\sigma_1(\overline{X}_{n-2})^2+\sigma_2(\overline{X}_{n-2})(\sigma_{n-2}(\overline{X}_{n-2})-1)=:f(n,\overline{X}_{n-2}).$$
\end{thm}

\begin{proof}
Let us observe that
\begin{align*}
&\ (\sigma_{n-2}(\overline{X}_{n-2})-1)(\sigma_{n}(\overline{X}_n)-\sigma_{2}(\overline{X}_n))\\
= &\ (\sigma_{n-2}(\overline{X}_{n-2})-1)\\
&\ \cdot\left[(\sigma_{n-2}(\overline{X}_{n-2})-1)x_{n-1}x_n-\sigma_{1}(\overline{X}_{n-2})(x_{n-1}+x_n)-\sigma_{2}(\overline{X}_{n-2})\right]\\
= &\ \left[(\sigma_{n-2}(\overline{X}_{n-2})-1)x_{n-1}-\sigma_{1}(\overline{X}_{n-2})\right]\left[(\sigma_{n-2}(\overline{X}_{n-2})-1)x_n-\sigma_{1}(\overline{X}_{n-2})\right]\\
&\ -f(n,\overline{X}_{n-2}).
\end{align*}
Thus, because $\sigma_{n-2}(\overline{X}_{n-2})\geq 2$ we see that $\overline{X}_n\in S(n)$ if and only if there are positive integers $d_{1}, d_{2}$ such that $d_{1}d_{2}=f(n,\overline{X}_{n-2})$ and the system of equations
$$
(\sigma_{n-2}(\overline{X}_{n-2})-1)x_{n-1}-\sigma_{1}(\overline{X}_{n-2})=d_{1},\quad (\sigma_{n-2}(\overline{X}_{n-2})-1)x_{n-1}-\sigma_{1}(\overline{X}_{n-2})=d_{2}.
$$
has a solution in integers $x_{n-1}, x_n$. Solving for $x_{n-1}, x_n$ we get the expressions from the statement.
\end{proof}

\begin{rem}
{\rm According to Lemma \ref{upperboundforxn-2} and Theorem \ref{sols}, it suffices to find all the possible values of $x_{n-1}$ and $x_n$ for all $\overline{X}_{n-2}$ such that $\sigma_{n-2}(\overline{X}_{n-2})\leq\binom{n}{2}$. Since there are only finitely many such $(n-2)$-tuples $\overline{X}_{n-2}$, the number of solutions of \eqref{maineq} is finite for each $n\in\N_{\geq 3}$.}
\end{rem}

Now we state the result that gives us an upper bound for unknown $x_n$, where $n\in\N_{\geq 3}$ is fixed.

\begin{thm}\label{mainbound}
Let $n\in\N_{\geq 3}$ and $m\in\N$ satisfies
$$
m=\sigma_{2}(\overline{X}_{n})=\sigma_{n}(\overline{X}_{n}).
$$
Then $m\leq n^2(3n-5)$ and $x_{n}\leq \frac{1}{2}n(3n-5)$.
\end{thm}

Before we prove Theorem \ref{mainbound} we need some preparation. First, if $\overline{X}_n\in S_i(n)$, then $i\geq 3$ by Lemma \ref{lowerboundforxn-2}. Thus, we may put
$$y_j=x_{n-i+j},\quad j\in\{1,\ldots,i-2\}$$
and
$$\overline{Y}_{i-2}=(y_1,\ldots,y_{i-2}).$$
Of course,
\begin{align}\label{4}
\overline{Y}_{i-2}\in\N_{\geq 2}^{i-2},\quad \overline{X}_{n-2}=(\overline{1}_{n-i},\overline{Y}_{i-2})
\end{align}
and
\begin{align}\label{5}
\sigma_{n-2}(\overline{X}_{n-2})=\sigma_{i-2}(\overline{Y}_{i-2}).
\end{align}
Let us estimate $\sigma_{1}(\overline{Y}_{i-2})$ and $\sigma_{2}(\overline{Y}_{i-2})$ from above in terms of $\sigma_{i-2}(\overline{Y}_{i-2})$.

\begin{lem}\label{estimate}
We have
$$\sigma_{1}(\overline{Y}_{i-2})\leq \frac{i-2}{2^{i-3}}\sigma_{i-2}(\overline{Y}_{i-2}), \quad \sigma_{2}(\overline{Y}_{i-2})\leq \frac{(i-2)(i-3)}{2^{i-3}}\sigma_{i-2}(\overline{Y}_{i-2}).$$
\end{lem}

\begin{proof}
Since $x_j\geq 2$ for each $j\in\{1,\ldots,i-2\}$, we have
$$y_k=\frac{\sigma_{i-2}(\overline{Y}_{i-2})}{\prod_{j\neq k}y_j}\geq\frac{\sigma_{i-2}(\overline{Y}_{i-2})}{2^{i-3}}$$
and
$$y_ky_l=\frac{\sigma_{i-2}(\overline{Y}_{i-2})}{\prod_{j\neq k,l}y_j}\leq\frac{\sigma_{i-2}(\overline{Y}_{i-2})}{2^{i-4}}$$
for $k,l\in\{1,\ldots,i-2\}$. As a result we get
$$\sigma_{1}(\overline{Y}_{i-2})=\sum_{k=1}^{i-2}y_k\leq (i-2)\frac{\sigma_{i-2}(\overline{Y}_{i-2})}{2^{i-3}}$$
and
$$\sigma_{2}(\overline{Y}_{i-2})=\sum_{1\leq k<l\leq i-2}y_ky_l\leq \binom{i-2}{2}\frac{\sigma_{i-2}(\overline{Y}_{i-2})}{2^{i-4}}.$$
The results follow.
\end{proof}

As an immediate consequence of Lemma \ref{estimate}, the formulae \eqref{sigmaid1}, \eqref{sigmaid2} and the conditions \eqref{4} and \eqref{5} we get the following.

\begin{cor}\label{estimate2}
If $n\in\N_{\geq 3}$ and $\overline{X}_n\in S_i(n)$, then
\begin{align}\label{6}
\sigma_{1}(\overline{X}_{n-2})\leq n-i+\frac{i-2}{2^{i-3}}\sigma_{n-2}(\overline{X}_{n-2})
\end{align}
and
\begin{align}\label{7}
\sigma_{2}(\overline{X}_{n-2})\leq \frac{1}{2}(n-i)(n-i-1)+\frac{(i-2)(n-3)}{2^{i-3}}\sigma_{n-2}(\overline{X}_{n-2}).
\end{align}
In particular, for each $\overline{X}_n\in S(n)$ we have
\begin{align}\label{sigma1bound}
\sigma_{1}(\overline{X}_{n-2})\leq n-3+\sigma_{n-2}(\overline{X}_{n-2})
\end{align}
and
\begin{align}\label{sigma2bound}
\sigma_{2}(\overline{X}_{n-2})\leq \frac{1}{2}(n-3)(n-4)+(n-3)\sigma_{n-2}(\overline{X}_{n-2}).
\end{align}
\end{cor}

\begin{proof}
The inequalities \eqref{sigma1bound} and \eqref{sigma2bound} follow as the expressions on the right hand side in \eqref{6} and \eqref{7} are clearly maximized for $i=3$.
\end{proof}

At this moment we are ready to give the proof of Theorem \ref{mainbound}.

\begin{proof}[Proof of Theorem \ref{mainbound}]
From Theorem \ref{sols} we know that $x_{n-1}=\frac{\sigma_1(\overline{X}_{n-2})+d_{1}}{\sigma_{n-2}(\overline{X}_{n-2})-1}$ and $x_n=\frac{\sigma_1(\overline{X}_{n-2})+d_{2}}{\sigma_{n-2}(\overline{X}_{n-2})-1}$, where $d_1d_2=f(n,\overline{X}_{n-2})$. Hence,
\begin{align*}
m &\ =\sigma_{n-2}(\overline{X}_{n-2})x_{n-1}x_n\\
&\ =\sigma_{n-2}(\overline{X}_{n-2})\frac{f(n,\overline{X}_{n-2})+(d_1+d_2)\sigma_1(\overline{X}_{n-2})+\sigma_1(\overline{X}_{n-2})^2}{(\sigma_{n-2}(\overline{X}_{n-2})-1)^2}.
\end{align*}
We consider two cases.

\bigskip

\textbf{Case I:} $\sigma_{n-2}(\overline{X}_{n-2})\leq n$. Then
\begin{align*}
x_n\leq \frac{\sigma_1(\overline{X}_{n-2})+f(n,\overline{X}_{n-2})}{\sigma_{n-2}(\overline{X}_{n-2})-1}=\frac{\sigma_1(\overline{X}_{n-2})(\sigma_1(\overline{X}_{n-2})+1)}{\sigma_{n-2}(\overline{X}_{n-2})-1}+\sigma_2(\overline{X}_{n-2}).
\end{align*}
Using inequalities \eqref{sigma1bound} and \eqref{sigma2bound} we obtain
\begin{align*}
x_n\leq &\ \frac{(\sigma_{n-2}(\overline{X}_{n-2})+n-3)(\sigma_{n-2}(\overline{X}_{n-2})+n-2)}{\sigma_{n-2}(\overline{X}_{n-2})-1}\\
&\ +\frac{1}{2}(n-3)(n-4)+(n-3)\sigma_{n-2}(\overline{X}_{n-2})\\
= &\ \sigma_{n-2}(\overline{X}_{n-2})+2n-4+\frac{(n-2)(n-1)}{\sigma_{n-2}(\overline{X}_{n-2})-1}\\
&\ +\frac{1}{2}(n-3)(n-4)+(n-3)\sigma_{n-2}(\overline{X}_{n-2})\\
= &\ (n-2)(\sigma_{n-2}(\overline{X}_{n-2})+2)+\frac{(n-2)(n-1)}{\sigma_{n-2}(\overline{X}_{n-2})-1}+\frac{1}{2}(n-3)(n-4).
\end{align*}
A simple analysis of the last expression treated as a function of $\sigma_{n-2}(\overline{X}_{n-2})\in [2,n]$ shows that this expression is maximized for $\sigma_{n-2}(\overline{X}_{n-2})\in\{2,n\}$ and attains the value
$$(n-2)(n+3)+\frac{1}{2}(n-3)(n-4)=\frac{1}{2}n(3n-5).$$

Now we estimate the value of $m$. Since $d_1+d_2\leq 1+f(n,\overline{X}_{n-2})$, we have
\begin{align*}
m\leq &\ \sigma_{n-2}(\overline{X}_{n-2})\cdot\frac{f(n,\overline{X}_{n-2})+(1+f(n,\overline{X}_{n-2}))\sigma_1(\overline{X}_{n-2})+\sigma_1(\overline{X}_{n-2})^2}{(\sigma_{n-2}(\overline{X}_{n-2})-1)^2}\\
= &\ \sigma_{n-2}(\overline{X}_{n-2})\cdot\frac{\sigma_1(\overline{X}_{n-2})+1}{\sigma_{n-2}(\overline{X}_{n-2})-1}\cdot\frac{\sigma_1(\overline{X}_{n-2})+f(n,\overline{X}_{n-2})}{\sigma_{n-2}(\overline{X}_{n-2})-1}
\end{align*}
From the estimation of $x_n$ we know that $\frac{\sigma_1(\overline{X}_{n-2})+f(n,\overline{X}_{n-2})}{\sigma_{n-2}(\overline{X}_{n-2})-1}\leq\frac{1}{2}n(3n-5)$, so it remains to bound $\sigma_{n-2}(\overline{X}_{n-2})\cdot\frac{\sigma_1(\overline{X}_{n-2})+1}{\sigma_{n-2}(\overline{X}_{n-2})-1}$ from above. By \eqref{sigma1bound} we have
\begin{align*}
\sigma_{n-2}(\overline{X}_{n-2})\cdot\frac{\sigma_1(\overline{X}_{n-2})+1}{\sigma_{n-2}(\overline{X}_{n-2})-1}\leq \sigma_{n-2}(\overline{X}_{n-2})\cdot\frac{\sigma_{n-2}(\overline{X}_{n-2})+n-2}{\sigma_{n-2}(\overline{X}_{n-2})-1}.
\end{align*}
An analysis of the derivative of the expression on the right hand side as a function of $\sigma_{n-2}(\overline{X}_{n-2})$ shows that this expression is maximized for $\sigma_{n-2}(\overline{X}_{n-2})\in\{2,n\}$ and attains the value $2n$. Summing up,
$$m\leq n^2(3n-5).$$

\bigskip

\textbf{Case II:} $\sigma_{n-2}(\overline{X}_{n-2})\geq n+1$. Since $x_{n-1}\geq 2$, we have
\begin{align*}
d_1= &\ (\sigma_{n-2}(\overline{X}_{n-2})-1)x_{n-1}-\sigma_{1}(\overline{X}_{n-2})\geq 2(\sigma_{n-2}(\overline{X}_{n-2})-1)-\sigma_{1}(\overline{X}_{n-2})\\
\geq &\ \sigma_{n-2}(\overline{X}_{n-2})-n+1\geq 2,
\end{align*}
where in the first inequality on the last line we used \eqref{sigma1bound}. Hence,
\begin{align*}
d_2=\frac{f(n,\overline{X}_{n-2})}{d_1}\leq\frac{\sigma_1(\overline{X}_{n-2})^2+\sigma_2(\overline{X}_{n-2})(\sigma_{n-2}(\overline{X}_{n-2})-1)}{\sigma_{n-2}(\overline{X}_{n-2})-n+1}.
\end{align*}
Consequently,
\begin{align*}
x_n\leq &\ \frac{\sigma_1(\overline{X}_{n-2})(\sigma_{n-2}(\overline{X}_{n-2})-n+1)+\sigma_1(\overline{X}_{n-2})^2+\sigma_2(\overline{X}_{n-2})(\sigma_{n-2}(\overline{X}_{n-2})-1)}{(\sigma_{n-2}(\overline{X}_{n-2})-1)(\sigma_{n-2}(\overline{X}_{n-2})-n+1)}\\
= &\ \frac{\sigma_1(\overline{X}_{n-2})(\sigma_{n-2}(\overline{X}_{n-2})-n+1+\sigma_1(\overline{X}_{n-2}))+\sigma_2(\overline{X}_{n-2})(\sigma_{n-2}(\overline{X}_{n-2})-1)}{(\sigma_{n-2}(\overline{X}_{n-2})-1)(\sigma_{n-2}(\overline{X}_{n-2})-n+1)}\\
\leq &\ \frac{\sigma_1(\overline{X}_{n-2})(2\sigma_{n-2}(\overline{X}_{n-2})-2)+\sigma_2(\overline{X}_{n-2})(\sigma_{n-2}(\overline{X}_{n-2})-1)}{(\sigma_{n-2}(\overline{X}_{n-2})-1)(\sigma_{n-2}(\overline{X}_{n-2})-n+1)}\\
\leq &\ \frac{2\sigma_1(\overline{X}_{n-2})+\sigma_2(\overline{X}_{n-2})}{\sigma_{n-2}(\overline{X}_{n-2})-n+1}\\
\leq &\ \frac{2(n-3)+2\sigma_{n-2}(\overline{X}_{n-2})+\frac{1}{2}(n-3)(n-4)+(n-3)\sigma_{n-2}(\overline{X}_{n-2})}{\sigma_{n-2}(\overline{X}_{n-2})-n+1}\\
= &\ \frac{\frac{1}{2}(n-3)n+(n-1)\sigma_{n-2}(\overline{X}_{n-2})}{\sigma_{n-2}(\overline{X}_{n-2})-n+1}= \frac{\frac{1}{2}(n-3)n+(n-1)^2}{\sigma_{n-2}(\overline{X}_{n-2})-n+1}+n-1\\
\leq &\ \frac{\frac{1}{2}(n-3)n+(n-1)^2}{2}+n-1=\frac{1}{4}(3n^2-3n-2),
\end{align*}
where in the third and fifth line we used \eqref{sigma1bound} and \eqref{sigma2bound}.

Now we estimate the value of $m$. Since $d_1\leq \sigma_{n-2}(\overline{X}_{n-2})-n+1$, the value of $d_1+d_2$ is maximized for $d_1= \sigma_{n-2}(\overline{X}_{n-2})-n+1=:d_1(\overline{X}_{n-2})$ and $d_2=\frac{f(n,\overline{X}_{n-2})}{\sigma_{n-2}(\overline{X}_{n-2})-n+1}=:d_2(\overline{X}_{n-2})$. Hence, we have
\begin{align*}
m\leq &\ \sigma_{n-2}(\overline{X}_{n-2})\cdot\frac{f(n,\overline{X}_{n-2})+(d_1(\overline{X}_{n-2})+d_2(\overline{X}_{n-2}))\sigma_1(\overline{X}_{n-2})+\sigma_1(\overline{X}_{n-2})^2}{(\sigma_{n-2}(\overline{X}_{n-2})-1)^2}\\
= &\ \sigma_{n-2}(\overline{X}_{n-2})\cdot\frac{\sigma_1(\overline{X}_{n-2})+d_1(\overline{X}_{n-2})}{\sigma_{n-2}(\overline{X}_{n-2})-1}\cdot\frac{\sigma_1(\overline{X}_{n-2})+d_2(\overline{X}_{n-2})}{\sigma_{n-2}(\overline{X}_{n-2})-1}.
\end{align*}
From the estimation of $x_n$ we know that $$\frac{\sigma_1(\overline{X}_{n-2})+d_2(\overline{X}_{n-2})}{\sigma_{n-2}(\overline{X}_{n-2})-1}\leq\frac{\frac{1}{2}(n-3)n+(n-1)^2}{\sigma_{n-2}(\overline{X}_{n-2})-n+1}+n-1,$$ so it remains to bound $\sigma_{n-2}(\overline{X}_{n-2})\cdot\frac{\sigma_1(\overline{X}_{n-2})+d_1(\overline{X}_{n-2})}{\sigma_{n-2}(\overline{X}_{n-2})-1}$ from above. By \eqref{sigma1bound} we have
\begin{align*}
&\ \sigma_{n-2}(\overline{X}_{n-2})\cdot\frac{\sigma_1(\overline{X}_{n-2})+d_1(\overline{X}_{n-2})}{\sigma_{n-2}(\overline{X}_{n-2})-1}\leq \sigma_{n-2}(\overline{X}_{n-2})\cdot\frac{2\sigma_{n-2}(\overline{X}_{n-2})-2}{\sigma_{n-2}(\overline{X}_{n-2})-1}\\
= &\ 2\sigma_{n-2}(\overline{X}_{n-2}).
\end{align*}
Hence,
$$m\leq 2\sigma_{n-2}(\overline{X}_{n-2})\left(\frac{\frac{1}{2}(n-3)n+(n-1)^2}{\sigma_{n-2}(\overline{X}_{n-2})-n+1}+n-1\right).$$
An analysis of the derivative of the expression on the right hand side as a function of $\sigma_{n-2}(\overline{X}_{n-2})\in \left[n+1,\binom{n}{2}\right]$ shows that this expression is maximized for $\sigma_{n-2}(\overline{X}_{n-2})=n+1$ and attains the value $\frac{1}{2}(n+1)(3n^2-3n-2)$.

\bigskip

Since $n\geq 3$, we have $\frac{1}{4}(3n^2-3n-2)\leq\frac{1}{2}n(3n-5)$ and $\frac{1}{2}(n+1)(3n^2-3n-2)<n^2(3n-5)$. Thus $x_n\leq \frac{1}{2}n(3n-5)$ and $m\leq n^2(3n-5)$ in any case.
\end{proof}

\begin{rem}
{\rm From the proofs of Corollary \ref{estimate2} and Theorem \ref{mainbound} we see that $x_n=\frac{1}{2}n(3n-5)$ and $m=n^2(3n-5)$ if and only if $i=3$ and $\{\sigma_{n-2}(\overline{X}_{n-2}),x_{n-1}\}=\{2,n\}$, i.e. $\overline{X}_n=\left(\overline{1}_{n-3},2,n,\frac{1}{2}n(3n-5)\right)$.}
\end{rem}

We are also able to estimate $x_{n-2}$.

\begin{lem}\label{C}
Let $n\in\N_{\geq 3}$ and $\overline{X}_{n}\in S(n)$. Assume that $x_{n-2}\geq 1+C(n-2)^{2/3}$ for some real number $C>0$. Then
\begin{align}\label{ineqC}
    C\leq\frac{1}{C}(n-2)^{-1/3}+\sqrt{\frac{3}{2C}(n-2)^{-1}+\frac{1}{C^2}(n-2)^{-2/3}+(n-2)^{-1/3}+\frac{1}{2C}}.
\end{align}
\end{lem}

\begin{proof}
By Theorem \eqref{sols} and the fact that $x_n\geq x_{n-1}\geq x_{n-2}\geq 1+C(n-2)^{2/3}$ we obtain the following chain of inequalities:
\begin{align*}
    &\ 1+C(n-2)^{2/3}\leq x_{n-2}\leq x_{n-1}\leq \frac{\sigma_1(\overline{X}_{n-2})+\sqrt{f(n,\overline{X}_{n-2})}}{\sigma_{n-2}(\overline{X}_{n-2})-1}\\
    &\ =\frac{\sigma_1(\overline{X}_{n-2})+\sqrt{\sigma_1(\overline{X}_{n-2})^2+\sigma_2(\overline{X}_{n-2})(\sigma_{n-2}(\overline{X}_{n-2})-1)}}{\sigma_{n-2}(\overline{X}_{n-2})-1}\\
    &\ \leq\frac{\sigma_{n-2}(\overline{X}_{n-2})-1+n-2}{\sigma_{n-2}(\overline{X}_{n-2})-1}\\
    &\ \quad +\frac{\sqrt{(\sigma_{n-2}(\overline{X}_{n-2})-1+n-2)^2+[\binom{n-2}{2}+(n-3)(\sigma_{n-2}(\overline{X}_{n-2})-1)](\sigma_{n-2}(\overline{X}_{n-2})-1)}}{\sigma_{n-2}(\overline{X}_{n-2})-1}\\
    &\ =1+\frac{n-2}{\sigma_{n-2}(\overline{X}_{n-2})-1}\\
    &\ \quad +\sqrt{\left(1+\frac{n-2}{\sigma_{n-2}(\overline{X}_{n-2})-1}\right)^2+\frac{n-3}{2}\frac{n-2}{\sigma_{n-2}(\overline{X}_{n-2})-1}+n-3}\\
    &\ \leq 1+\frac{n-2}{x_{n-2}-1}+\sqrt{\left(1+\frac{n-2}{x_{n-2}-1}\right)^2+\frac{n-3}{2}\frac{n-2}{x_{n-2}-1}+n-3}\\
    &\ \leq 1+\frac{n-2}{C(n-2)^{2/3}}+\sqrt{\left(1+\frac{n-2}{C(n-2)^{2/3}}\right)^2+\frac{n-3}{2}\frac{n-2}{C(n-2)^{2/3}}+n-3}\\
    &\ \leq 1+\frac{1}{C}(n-2)^{1/3}\\
    &\ \quad +\sqrt{1+\frac{2}{C}(n-2)^{1/3}+\frac{1}{C^2}(n-2)^{2/3}+\frac{1}{2C}(n-2)^{4/3}-\frac{1}{2C}(n-2)^{1/3}+n-3}\\
    &\ =1+\frac{1}{C}(n-2)^{1/3}+\sqrt{\frac{3}{2C}(n-2)^{1/3}+\frac{1}{C^2}(n-2)^{2/3}+n-2+\frac{1}{2C}(n-2)^{4/3}}.
\end{align*}
After reduction the 1's and dividing by $(n-2)^{2/3}$ we get the the statement.
\end{proof}

\begin{cor}\label{xn2ineq}
Let $n\in\N_{\geq 3}$ and $\overline{X}_{n}\in S(n)$. Then we have
\begin{align}\label{ineqxS3}
    x_{n-2}\leq 1+\lfloor 2(n-2)^{2/3}\rfloor ,
\end{align}
where the equality holds only if $n=3$. Moreover, for each $C>2^{-1/3}$ there are only finitely many values of $n\in\N_{\geq 3}$ such that there exists $\overline{X}_{n}\in S(n)$ with $x_{n-2}>\lceil C(n-2)^{2/3}\rceil$.
\end{cor}

\begin{proof}
If $x_{n-2}=1+C(n-2)^{2/3}$ for some $C>0$ then by Lemma \ref{C} we have
\begin{align}\label{ineqC2}
 C\leq\frac{1}{C}+\sqrt{\frac{3}{2C}+\frac{1}{C^2}+1+\frac{1}{2C}}
\end{align}
since $n\geq 3$. The left hand side of \eqref{ineqC2} is an increasing function of $C$ while the right hand side is a decreasing one. Moreover, the equality holds for $C=2$. Hence, \eqref{ineqC2} holds if and only if $C\leq 2$. Thus $x_{n-2}\leq 1+2(n-2)^{2/3}$. As $x\in\Z$, we have $x\leq 1+\lfloor 2(n-2)^{2/3}\rfloor$. Moreover, if $n>3$ then $C=2$ does not satisfy inequality \eqref{ineqC}. Therefore the equality in \eqref{ineqxS3} holds only if $n=3$.

We are left with the proof of the "moreover" part. Let $C>2^{-1/3}$. Assume by contrary that there are infinitely many values of $n\in\N$ such that there exists $\overline{X}_{n}\in S(n)$ with $x_{n-2}\geq 1+C(n-2)^{2/3}$. Tending with $n$ to $+\infty$ in inequality \eqref{ineqC} we obtain
$$C\leq\frac{1}{\sqrt{2C}},$$
which means that $C\leq\frac{1}{\sqrt[3]{2}}$. This is a contradiction with our assumption that $C>2^{-1/3}$.
\end{proof}

\begin{rem}
{\rm Let $\overline{X}_{n}\in S(n)$. If the rate of the magnitude of $x_{n-2}$ is $2^{-1/3}n^{2/3}+o(n^{2/3})$ then it is also the rate of the magnitude of $x_{n-1}$ and $x_n$. Indeed, from the proof of Lemma \ref{C} we know that
\begin{align*}
    &\ 2^{-1/3}(n-2)^{2/3}+o((n-2)^{2/3})\leq x_{n-2}\leq x_{n-1}=\frac{\sigma_1(\overline{X}_{n-2})+d_1}{\sigma_{n-2}(\overline{X}_{n-2})-1}\\
    &\ \leq\frac{\sigma_1(\overline{X}_{n-2})+\sqrt{f(n,\overline{X}_{n-2})}}{\sigma_{n-2}(\overline{X}_{n-2})-1}\\
    &\ \leq 1+2^{1/3}(n-2)^{1/3}\\
    &\ \quad +\sqrt{3\cdot 2^{-2/3}(n-2)^{1/3}+2^{2/3}(n-2)^{2/3}+n-2+2^{-2/3}(n-2)^{4/3}}\\
    &\ =2^{-1/3}n^{2/3}+o(n^{2/3}),
\end{align*}
where $\sigma_{n-2}(\overline{X}_{n-2})=x_{n-2}$ (i. e. $\overline{X}_{n-2}=(\overline{1}_{n-3},x_{n-2})$) and $d_1\sim\sqrt{f(n,\overline{X}_{n-2})}=\sqrt{f(n,x_{n-2})}$, $n\to +\infty$. This means that also $d_2\sim\sqrt{f(n,x_{n-2})}$ and analogously we compute that
$$x_n=\frac{\sigma_1(\overline{X}_{n-2})+d_1}{\sigma_{n-2}(\overline{X}_{n-2})-1}=2^{-1/3}n^{2/3}+o(n^{2/3}).$$

In fact there is an infinite family of quadruples $(n,x,y,z)$ of positive integers such that $(\overline{1}_{n-3}, x, y, z)\in S_{3}(n)$ and $x\sim y\sim z\sim 2^{-1/3}n^{2/3}$, $n\to +\infty$. Namely,
\begin{align*}
    n= &\ 4(4k^3+2k^2+2k-2)^3+2,\\
    x= &\ 2(4k^3+2k^2+2k-2)\\
    &\ +32k^6+32k^5+32k^4-16k^3-8k^2-10k+8,\\
    y= &\ 2(4k^3+2k^2+2k-2)^2+1,\\
    z= &\ (4k^3+2k^2+2k-2)(8k^3+4k^2+6k-1)+1,
\end{align*}
where $k\in\N$ is sufficiently large. In the above family we have $y=1+2^{-1/3}(n-2)^{2/3}$. However, it is possible to give a family of quadruples $(n,x,y,z)$ such that $(\overline{1}_{n-3}, x, y, z)\in S_{3}(n)$ and $x=1+2^{-1/3}(n-2)^{2/3}$. Such a family is
\begin{align*}
    n= &\ 4(4k^3-10k^2+10k-6)^3+2,\\
    x= &\ 2(4k^3-10k^2+10k-6)^2+1,\\
    y= &\ (4k^3-10k^2+10k-6)(8k^3-20k^2+22k-11)+1,\\
    z= &\ 2(4k^3-10k^2+10k-6)\\
    &\ +32k^6-150k^5+352k^4-464k^3+392k^2-202k+58,
\end{align*}
where $k\in\N$ is sufficiently large.}
\end{rem}

Gathering all the results above we are able to enumerate all elements of $S(n)$ for $n\leq 16$. More precisely, we have the following.

\begin{thm}\label{smallsolutions}
We have the following equalities of sets:
{\small
\begin{align*}
S(3)&=\{(2, 3, 6), (2, 4, 4), (3, 3, 3)\};\\
S(4)&=\{(1, 2, 4, 14), (2, 2, 2, 6)\};\\
S(5)&=\{(1, 1, 2, 5, 25), (1, 1, 2, 7, 11), (1, 1, 3, 3, 22), (1, 1, 3, 4, 9), (1, 2, 2, 2, 18),\\
    &\quad\;\; (1, 2, 2, 4, 4), (2, 2, 2, 2, 3)\};\\
S(6)&=\{(1, 1, 1, 2, 6, 39), (1, 1, 1, 2, 7, 22), (1, 1, 1, 3, 4, 18), (1, 1, 1, 3, 6, 8)\};\\
S(7)&=\{(1, 1, 1, 1, 2, 7, 56), (1, 1, 1, 1, 2, 8, 31), (1, 1, 1, 1, 2, 11, 16), (1, 1, 1, 1, 3, 4, 46),\\
    &\quad\;\; (1, 1, 1, 1, 3, 6, 12), (1, 1, 1, 1, 4, 6, 7), (1, 1, 1, 2, 2, 3, 20)\};\\
S(8)&=\{(\overline{1}_{5}, 2, 8, 76), (\overline{1}_{5}, 2, 10, 30), (\overline{1}_{5}, 4, 4, 22), (\overline{1}_{4}, 2, 2, 3, 50), (\overline{1}_{3}, 2, 2, 2, 3, 5)\};\\
S(9)&=\{(\overline{1}_{6}, 2, 9, 99), (\overline{1}_{6}, 2, 15, 21), (\overline{1}_{6}, 3, 5, 78), (\overline{1}_{67}, 3, 6, 29), (\overline{1}_{6},, 3, 8, 15)\};\\
S(10)&=\{(\overline{1}_{7}, 2, 10, 125), (\overline{1}_{7}, 2, 11, 67), (\overline{1}_{8}, 2, 13, 38), (\overline{1}_{7}, 3, 6, 51), (\overline{1}_{7}, 3, 7, 28), (\overline{1}_{7}, 4, 4, 93),\\
     &\quad\;\; (\overline{1}_{7}, 4, 5, 26), (\overline{1}_{7}, 6, 7, 7), (\overline{1}_{6}, 2, 3, 3, 21), (\overline{1}_{6}, 3, 3, 3, 8)\};\\
S(11)&=\{(\overline{1}_{8}, 2, 11, 154), (\overline{1}_{8}, 2, 12, 82), (\overline{1}_{8},2,13,58),(\overline{1}_{8},2,14,46),(\overline{1}_{8},2,16,34),(\overline{1}_{8},2,18,28),\\
    &\quad \;\;(\overline{1}_{8},2,19,26),(\overline{1}_{8},2,22,22),(\overline{1}_{8},3,6,118),(\overline{1}_{8},3,7,43),(\overline{1}_{8},3,8,28), (\overline{1}_{8},3,10,18),\\
    &\quad \;\;(\overline{1}_{8},3,13,13),(\overline{1}_{8},4,5,40),(\overline{1}_{8},4,6,22),(\overline{1}_{8},4,7,16), (\overline{1}_{8},4,8,13),(\overline{1}_{8},4,10,10),\\
    &\quad \;\;(\overline{1}_{8},5,5,19),(\overline{1}_{8},6,6,10),(\overline{1}_{8},7,7,7), (\overline{1}_{7},2,2,4,97),(\overline{1}_{7},2,2,5,27),(\overline{1}_{7},2,2,6,17),\\
    &\quad\;\; (\overline{1}_{7},2,2,7,13),(\overline{1}_{6},2,2,2,3,11)\};\\
S(12)&=\{(\overline{1}_{9},2,12,186),(\overline{1}_{9},2,16,46),(\overline{1}_{9},2,18,36),(\overline{1}_{9},4,6,30),(\overline{1}_{9},4,8,16),\\
   &\quad\;\; (\overline{1}_{9},6,6,12),(\overline{1}_{8},2,3,4,18),(\overline{1}_{8},2,4,4,10),(\overline{1}_{8},2,4,6,6),(\overline{1}_{7},2,2,2,2,101)\};\\
S(13)&=\{(\overline{1}_{10},2,13,221),(\overline{1}_{10},2,23,31),(\overline{1}_{10},3,7,166),(\overline{1}_{10},3,12,21),(\overline{1}_{10},4,5,155),\\
   &\quad\;\;(\overline{1}_{10},5,5,34),(\overline{1}_{9},2,3,3,129),(\overline{1}_{9},3,3,3,16),(\overline{1}_{8},2,2,4,4,4)\};\\
S(14)&=\{(\overline{1}_{11},2,14,259),(\overline{1}_{11},2,15,136),(\overline{1}_{11},2,16,95),(\overline{1}_{11},2,19,54),(\overline{1}_{11},3,8,100),\\
     &\quad\;\;(\overline{1}_{11},3,10,38),(\overline{1}_{11},4,6,63),(\overline{1}_{11},4,7,34),(\overline{1}_{10},2,2,5,159),(\overline{1}_{9},2,2,3,5,5)\};\\
S(15)&=\{(\overline{1}_{12},2,15,300),(\overline{1}_{12},2,16,157),(\overline{1}_{12},2,25,40),(\overline{1}_{12},2,27,36),(\overline{1}_{12},3,8,222),\\
     &\quad \;\;(\overline{1}_{12},3,9,79),(\overline{1}_{12},3,13,27),(\overline{1}_{12},3,14,24),(\overline{1}_{12},4,6,105),(\overline{1}_{12},4,13,14),\\
     &\quad\;\;(\overline{1}_{11},2,3,4,45),(\overline{1}_{11},2,3,7,12),(\overline{1}_{10},2,2,2,3,33)\};\\
S(16)&=\{(\overline{1}_{13},2,16,344),(\overline{1}_{13},2,22,62),(\overline{1}_{13},4,6,232),(\overline{1}_{13},4,8,38),(\overline{1}_{13},8,8,10),\\
     &\quad\;\;(\overline{1}_{12},2,2,6,107),(\overline{1}_{12},2,2,7,46),(\overline{1}_{12},2,3,6,18),(\overline{1}_{11},2,2,2,3,46)\}.
\end{align*}
}
In particular, we have the following values of $|S(n)|$.
\begin{equation*}
\begin{array}{|c|cccccccccccccc|}
\hline
 n  & 3 & 4  & 5 & 6 & 7 & 8  & 9  & 10 & 11 & 12 & 13 & 14 & 15 & 16  \\
  \hline
 |S(n)| & 3 & 2 & 7 & 4 & 7 & 5  & 5 &10  & 26  & 10  & 9  & 10  & 13 & 9 \\
  \hline
\end{array}
\end{equation*}

\begin{center}{\rm Table 1. The number of element of $S(n)$ for $n\leq 16$.}\end{center}

Moreover, for each $n\in\N_{\geq 3}$ we have $S(n)\neq \emptyset$.
\end{thm}
\begin{proof}
To find all elements of $S(n)$ we used a simple computer search. More precisely, we solved the equation (\ref{maineq}) for $x_{n}$ and get
$$
x_{n}=\frac{\sigma_{2}(\overline{X}_{n-1})}{\sigma_{n-1}(\overline{X}_{n-1})-\sigma_{1}(\overline{X}_{n-1})}.
$$
Next, for given $n$, we computed, via Lemma \ref{inumber}, the number of 1's in the solution of (\ref{maineq}) and used the bounds
$$
x_{n-1}\leq \frac{1}{2}n(3n-5),\quad x_{n-2}\leq  1+\lfloor 2(n-2)^{2/3}\rfloor
$$
to check all possibilities $x_{1}\leq x_{2}\leq\ldots\leq x_{n-2}\leq x_{n-1}$ for which $x_{n}$ computed above is a positive integer.

To get the last statement it is enough to note that for $n\geq 3$ we have that $\left(\overline{1}_{n-3}, 2, n, \frac{1}{2}n(3n-5)\right)\in S(n)$. All these computations took less then 2 hours on a laptop with 32 GB of RAM and i7 type processor.
\end{proof}

\begin{cor}
We have
$$
\limsup_{n\rightarrow +\infty}\max_{\overline{X}_{n}\in S(n)}\frac{x_{n}}{x_{n-1}}=+\infty.
$$
\end{cor}
\begin{proof}
Because for each $n\in\N_{\geq 3}$ we have $\left(1,\ldots, 1, 2, n, n(3n-5)/2\right)\in S(n)$ we get
$$
\limsup_{n\rightarrow +\infty}\max_{\overline{X}_{n}\in S(n)}\frac{x_{n}}{x_{n-1}}\geq \limsup_{n\rightarrow +\infty}\frac{n(3n-5)}{2n}=+\infty,
$$
and hence the result.
\end{proof}

\section{Analysis of $S_{3}(n)$ and applications}\label{sec3}

From Theorem \ref{sols} we know that $(\overline{1}_{n-3},x,y,z)\in S_{3}(n)$ if and only if
\begin{equation}\label{yzsol}
y=\frac{n+d_{1}-2}{x-1}+1,\quad z=\frac{n+d_{2}-2}{x-1}+1,
\end{equation}
where
\begin{equation*}
d_{1}d_{2}=\frac{1}{2}(n-2)((x+1)n+2x^2-3x-3)=f(n,(\overline{1}_{n-3},x))=:f(n,x).
\end{equation*}
It is clear that $|S_{3}(n)|$ is finite because from Corollary \ref{xn2ineq} we know that $x\leq 1+\lfloor 2(n-2)^{2/3}\rfloor$. Thus, we can even present a crude upper bound
$$
|S_{3}(n)|\leq \sum_{2\leq x\leq 1+\lfloor 2(n-2)^{2/3}\rfloor}\tau(f(n,x)),
$$
where $\tau(m)$ is the number of positive integer divisors of $m$.

\begin{rem}
{\rm Using the above characterization we computed the set $S_{3}(n)$ for each $n\leq 300$. More precisely, for given $n$ and each $x\leq 1+\lfloor 2(n-2)^{2/3}\rfloor$ we computed the set
$$
D(n,x):=\{d:\;d|f(n,x)\}
$$
of positive integer divisors of $f(n,x)$. Next, for any given $d\in D(n,x)$ such that $d\leq \sqrt{f(n,x)}$ we checked whether the numbers $y, z$ given by (\ref{yzsol}), where $d_1=d$ and $d_2=f(n,x)/d$, are integers.

In the considered range the value of $|S_{3}(n)|$ attains maximum equal to 213 for $n=299$.} 

%
\end{rem}

From Theorem \ref{smallsolutions} we know that $S(n)$ is nonempty. We prove that $|S(n)|\geq 3$. More precisely, the following is true.

\begin{thm}
\begin{enumerate}
\item For $n=3$ and each $n\in\N_{\geq 5}$ we have $|S_3(n)|\geq 3$.
\item We have
$$
|S_{3}(n)|\geq \frac{1}{2}\tau\left(\frac{1}{2}(n-2)(3n-1)\right).
$$
In particular $\limsup_{n\rightarrow +\infty}|S_{3}(n)|=\limsup_{n\rightarrow +\infty}|S(n)|+\infty$.
\end{enumerate}
\end{thm}
\begin{proof}
To get the first statement we note that $|S_{3}(3)|=3, |S_{3}(5)|=3, |S_{3}(6)|=4, |S_{3}(7)|=6, |S_{3}(8)|=3$, i.e., we can assume that $n\geq 9$. We observe that for each $n\geq 5$, the $n$-tuple $(\overline{1}_{n-3}, 2, n, \frac{1}{2}n(3n-5))\in S_{3}(n)$. If $n\equiv 1\pmod*{2}$ we have two additional $n$-tuples
$$
\left(\overline{1}_{n-3}, 2, 2n-3, \frac{1}{2}(5n-3)\right),\quad \left(\overline{1}_{n-3}, 3, n-1, \frac{3}{2}(n+1)\right).
$$
Similarly, if $n\equiv 0\pmod*{2}$ we have the following additional solutions
$$
\left(\overline{1}_{n-3}, 2, \frac{1}{2}(3n-4), 2(2n-1)\right),\quad \left(\overline{1}_{n-3}, 4, \frac{1}{2}n, 2(n+3)\right).
$$

The second statement is a consequence of the following reasoning. Take $x=2$ in (\ref{yzsol}). Then the corresponding values of $y, z$ are
$$
y=n+d_{1}-1,\quad z=n+\frac{f(n,2)}{d_{1}}-1
$$
and are integers. Thus, the number of different pairs satisfying $y\leq z$ is equal to $\left\lceil\frac{1}{2}\tau(f(n,2))\right\rceil$ and we get the required inequality. Moreover, if we take $n\equiv 2\pmod*{2q_{1}\cdot\ldots\cdot q_{k}}$, where $k\in\N_{+}$ and $q_{1}<\ldots<q_{k}$ are primes, then $\frac{1}{2}\tau(f(n,2))\geq k$ and $\limsup_{n\rightarrow +\infty}|S_{3}(n)|=+\infty$. Thus $\limsup_{n\rightarrow +\infty}|S(n)|=+\infty$.


\end{proof}

\begin{rem}
{\rm Let us note that if we take $x=3$ in (\ref{yzsol}) then the necessary and sufficient condition for $n$ in order to get $(\overline{1}_{n-3}, 3, y, z)\in S(n)$ is $n\not\equiv 0\pmod*{4}$.
}
\end{rem}
We are in position to compute the value of $\liminf_{\overline{X}_{n}\in S(n)}x_{n}/x_{n-1}$. More precisely, the following is true.

\begin{thm}\label{xneqxn}
There are infinitely many pairs $(n,X)$ of positive integers such that $(\overline{1}_{n-3}, 2, X, X)\in S_{3}(n)\subset S(n)$. In particular, we have
$$
\liminf_{n\rightarrow +\infty}\min_{\overline{X}_{n}\in S(n)}\frac{x_{n}}{x_{n-1}}=1.
$$
\end{thm}
\begin{proof}
To get the result we show that there are infinitely many values of $n$ such that there is an integer $X$ such that $(\overline{1}_{n-3}, 2, X, X)\in S_{3}(n)\subset S(n)$. We know that this is enough to show that $d^2=f(n,2)$ has infinitely many solutions in positive integers. Then $X=n+d-1$ is a solution we are looking for. The equation $d^2=f(n,2)$ is equivalent with Pell type equation $Y^2-24d^2=25$, where $Y=6n-7$. Using standard method, one can check that for each $m\in\N$ the pair $(d_{m},Y_{m})$, where $d_{m}=f_{m}(1, 10), Y_{m}=f_{m}(5, 49)$ for $m=0, 1, \ldots$ and $f_{m}(a,b)=10f_{m-1}(a,b)-f_{m-2}(a,b)$ with $f_{0}(a,b)=a, f_{1}(a,b)=b$, satisfies $Y_{m}^2-24d_{m}^2=1$ and hence the pair $(5d_{m}, 5Y_{m})$ solves the equation $Y^2-24d^2=25$. To get integer values of $n$ we need to take $m\equiv 1\pmod*{2}$. With $n$ chosen in this way the equation $\sigma_{2}(\overline{X}_{n})=\sigma_{n}(\overline{X}_{n})$ has the solution
$$
\overline{X}_{n(m)}=(\overline{1}_{n(m)-3}, 2, n(m)+d_{2m+1}-1,  n(m)+d_{2m+1}-1), m=0, 1, \ldots,
$$
satisfying $x_{n-1}=x_{n}$ and hence the result.
\end{proof}

As an example, we computed the values of $n=n(m)$ and the corresponding solution $x_n$ for $m=0, 1, \ldots, 5$.
\begin{equation*}
\begin{array}{|c|llllll|}
\hline
 m & 0 & 1 & 2 & 3 & 4 & 5 \\
 \hline
 n & 42 & 4002 & 392042 & 38416002 & 3764376042 & 368870436002 \\
 x_{n} & 91 & 8901 & 872191 & 85465801 & 8374776291 & 820642610701 \\
 \hline
\end{array}
\end{equation*}
\begin{center}
Table 2. Some values of $n$ such that there is an $x_{n}$ satisfying $(\overline{1}_{n-3}, 2, x_{n}, x_{n})\in S(n)$.
\end{center}



In Theorem \ref{xneqxn} we proved that limes inferior of $\op{min}\{x_{n}/x_{n-1}:\;\overline{X}_{n}\in S(n)\}$ is equal to 1. In this context it is natural to investigate the set
$$
D:=\left\{\frac{x_{n}}{x_{n-1}}:\;\overline{X}_{n}\in S(n), n\in\N_{\geq 3}\right\}\subset \Q
$$
and the set
$$
D_{i}:=\left\{\frac{x_{n}}{x_{n-1}}:\;\overline{X}_{n}\in S_{i}(n)\right\}\subset \Q
$$
where $i\in\N_{\geq 3}$, and ask whether $D$ and $D_{i}$ are dense in the set $[1,+\infty)$.

We prove the following.

\begin{thm}
The sets $D_3$ and $D$ are dense in $[1,+\infty)$.
\end{thm}

\begin{proof}
It suffices to show that $D_3$ is dense in $[1,+\infty)$ as $D\supset D_3$.

One can check that the set of limit points of $D_{3}$ contains the set
$$
\left\{\frac{\sqrt{24ab}+b}{\sqrt{24ab}+a}:\;a,b\in\N_{+},\;24ab\;\mbox{is not a square}\right\}
$$
Indeed, if $(\overline{1}_{n-3}, x, x_{n-1}, x_{n})\in S_{3}(n)$ and we take $x=2$ in (\ref{yzsol}) then $x_{n-1}=n+d_{1}-1, x_{n}=n+d_{2}-1$ with $d_{1}d_{2}=f(n,2)$. If we put $d_{1}=at$ and $d_{2}=bt$ then we deal with Pell type equation $abt^2=f(n,2)$ or equivalently
\begin{align}\label{pell}
    u^2-24abt^2=25,
\end{align}
where $u=6n-7$. From the theory of Pell equations we know that this equation has infinitely many solutions in positive integers provided that $24ab$ is not a square of an integer. Moreover, since $(u,t)=(5,0)$ is a solution of \eqref{pell}, there are infinitely many solutions of \eqref{pell} with $u\equiv 5\pmod{6}$, which ensures the integrality of $n$. Hence, $(d_{1}, d_{2}, n)$ satisfies
$$
d_{1}\sim a\frac{6n-7}{\sqrt{24ab}}, \quad d_{2}\sim b\frac{6n-7}{\sqrt{24ab}}, \quad n\to +\infty.
$$
In consequence
\begin{align*}
    &\ \lim_{n\rightarrow +\infty}\frac{x_{n}}{x_{n-1}}=\lim_{n\rightarrow +\infty}\frac{n+d_{2}-1}{n+d_{1}-1}=\lim_{n\rightarrow +\infty}\frac{n+b\frac{6n-7}{\sqrt{24ab}}-1}{n+a\frac{6n-7}{\sqrt{24ab}}-1}=\frac{\sqrt{24ab}+6b}{\sqrt{24ab}+6a}\\
    =&\ \frac{\sqrt{6ab}+3b}{\sqrt{6ab}+3a}=\frac{\sqrt{6\frac{b}{a}}+3\frac{b}{a}}{\sqrt{6\frac{b}{a}}+3}.
\end{align*}
Since the function $[1,+\infty)\ni x\mapsto\frac{\sqrt{6x}+x}{\sqrt{6x}+1}\in [1,+\infty)$ is a continuous surjection, we infer that the set
$$
\left\{\frac{\sqrt{24ab}+b}{\sqrt{24ab}+a}:\;a,b\in\N_{+},\;24ab\;\mbox{is not a square}\right\}
$$
is dense in $[1,+\infty)$. Hence, the closure of $D_3$ is the whole interval $[1,+\infty)$.
\end{proof}

\section{Questions and conjectures}\label{sec4}

In this section we collect some questions and conjectures which appeared during various stages of our investigations.

The most important (and probably the most difficult) is the following

\begin{ques}
What is the order of growth of the number $|S(n)|$? Is it true that there is $\varepsilon >0$ such that $|S(n)|>n^{\varepsilon}$? In particular, is the equality
$$
\liminf_{n\rightarrow +\infty}|S(n)|=+\infty
$$
true?
\end{ques}

This is a difficult question. We believe that the inequality $|S(n)|>\log n$ is true. In this context one can also ask the following questions.

\begin{ques}
What is the average order of the function $|S(n)|$?
\end{ques}

\begin{ques}
Is the function $|S(n)|$ eventually increasing?
\end{ques}

\bigskip

For given $n\in\N_{\geq 3}$ and $\overline{X}_{n}\in S(n)$ let us put $X_{n}=\{x_{1},\ldots, x_{n}\}$ and
$$
M(n):=\op{min}\{|X_{n}|:\;\overline{X}_{n}\in S(n)\}.
$$

From emptiness of $S_{i}(n)$ for $i=0, 1, 2$ and our investigations concerning the structure of the set $S_{3}(n)$, or to be more precise the proof of Theorem \ref{xneqxn}, we know that
$$
\liminf_{n\rightarrow+\infty}M(n)\leq 3.
$$

Let $n\geq 6$. Then $M(n)=1$ is impossible by Corollary \ref{ones}. Moreover, $M(n)=2$ if and only if there is $k\in\{1,\ldots, n-1\}$ and $x\in\N_{\geq 2}$ such that
\begin{equation}\label{M2}
\sigma_{2}(\overline{1}_{n-k},\overline{x}_{k})=\sigma_{n}(\overline{1}_{n-k},\overline{x}_{k}),
\end{equation}
where $\overline{x}_{k}=(x,\ldots,x)$, where we have $k$-occurrences of $x$.

Equivalently, we deal with the equation
\begin{equation*}
y^2=8 x^k+4 k x^2-4 k x+1=:P_{k}(x),
\end{equation*}
where $y=2n+2k(x-1)-1.$ Let $C_{k}$ denotes the curve defined by the equation above. One can check, using \cite[Theorem 3.1]{Ga}, that the discriminant of $P_{k}(x)$ is non-zero, and thus the polynomial $P_{k}$ has no multiple roots (moreover, since the Newton polygon of $P_{k}$ with respect to $2$-adic valuation has only one slope, equal to $3/k$, we see that for each $k\in\N_{\geq 3}$ not divisible by $3$ the polynomial $P_{k}$ is irreducible over the field $\Q_2$ of $2$-adic numbers and thus irreducible over $\Q$). In consequence, the genus of the curve $C_{k}$ is equal to $\lfloor (k-1)/2\rfloor$. For $k=3$ the curve $C_{k}$ is an elliptic curve and we have $C_{k}(\Q)\simeq \Z\times\Z_{3}$, where the infinite part is generated by the point $P=(8,24)$, and the torsion part is generated by $T=(0,1)$. Using standard methods one can check that the only integer points on $C_{3}$ which lead to solutions of the equation (\ref{M2}) are: $(x, y)=(3,17)$ which leads to $(k, n, x)=(3, 3, 3)$ (and the value $M(3)=1$) and $(x,y)=(7,57)$ which leads to $(k, n, x)=(3, 11, 7)$.

Similar analysis can be performed in the case $k=4$. For example, one can use the procedure {\tt IntegralQuarticPoints} of the Magma computational package (\cite{Mag}) and find that all the integral points on $C_{4}$ are:
$$(1,\pm 3), (0, \pm 1), (7, \pm 141), (-2, \pm 15), (-3, \pm 29), (172, \pm 83679).$$
The point $(7, 141)$ gives the solution $(k, n, x)=(4, 47, 7)$ and the point $(172, 83679)$ gives the solution $(k, n, x)=(4, 41156, 172)$ of the equation (\ref{M2}).
If $k\geq 5$ then the curve $C_{k}$ is of genus $\geq 2$, and from the Faltings theorem we know that the set $C_{k}(\Q)$ is finite. From numerical calculations we expect that besides the trivial point $(0,\pm 1)$ there is no additional integral points on $C_{k}$. We thus formulate the following.

\begin{conj}
The only values of $n\in\N_{\geq 3}$ such that $M(n)=2$ are $n=3, 4, 5, 11, 41156$. In particular,  $\liminf_{n\rightarrow+\infty}M(n)=3$.
\end{conj}

Since $\left(\overline{1}_{n-3},2,n,\frac{1}{2}n(3n-5)\right)\in S(n)$ for each $n\in\N_{\geq 3}$, we know that
$$\limsup_{n\to +\infty} M(n)\leq 4.$$
We suppose that the condition $M(n)=3$ is rarely satisfied.

\begin{conj}
We have $\limsup_{n\to +\infty} M(n)=4.$
\end{conj}

In this context we also prove the following.

\begin{thm}
For each $m\in\N_{\geq 2}$ and each sequence $\overline{Y}_{m}=(y_{1}, y_{2},\ldots, y_{m})\in\N_{\geq 2}^{m}\backslash\{(2,2),(2,3)\}$ there are $k, Y\in\N_{+}$ such that for $n=k+m+1$ we have $(\overline{1}_{k},\overline{Y}_{m},Y)\in S(n)$. In particular,
$$
\limsup_{n\rightarrow+\infty}\max_{\overline{X}_n\in S(n)}|X_n|=+\infty.
$$
\end{thm}
\begin{proof}
To find suitable values of $k, Y$ we consider the equation
$$
\sigma_{2}(\overline{1}_{k},\overline{Y}_{m},Y)=\sigma_{n}(\overline{1}_{k},\overline{Y}_{m},Y)=\sigma_{m}(\overline{Y}_{m})Y,
$$
which is equivalent with the equation
$$
\sigma_{2}(\overline{1}_{k})+\sigma_{2}(\overline{Y}_{m})+\sigma_{1}(\overline{1}_{k})\sigma_{1}(\overline{Y}_{m})+(\sigma_{1}(\overline{1}_{k})+\sigma_{1}(\overline{Y}_{m}))Y=\sigma_{m}(\overline{Y}_{m})Y,
$$
i.e.
$$
\binom{k}{2}+\sigma_{2}(\overline{Y}_{m})+k\sigma_{1}(\overline{Y}_{m})=(\sigma_{m}(\overline{Y}_{m})-k-\sigma_{1}(\overline{Y}_{m}))Y.
$$
Thus, by taking $k=\sigma_{m}(\overline{Y}_{m})-\sigma_{1}(\overline{Y}_{m})-1$ we get the corresponding value of $Y=\binom{k}{2}+\sigma_{2}(\overline{Y}_{m})+k\sigma_{1}(\overline{Y}_{m})$. Under our assumption on $\overline{Y}_{m}$ it is clear that $k>0$ (for an easy proof of this fact see \cite{Eck}) and we get required solution.

To get the second statement it is enough to take $\overline{Y}_{m}=(2,\ldots, m+1)$. Then $k=(m+1)!-\binom{m+2}{2}$ and $n=(m+1)!-\binom{m+2}{2}+m+1$. The corresponding element $\overline{X}_{n}\in S(n)$ satisfies $|X_{n}|=m+2$ and hence the result.
\end{proof}

\begin{prob}
For which $a\in\N_{\geq 4}$ the equation $\max_{\overline{X}_n\in S(n)}|X_n|=a$ has infinitely many solutions?
\end{prob}

Our expectation is that for each $a\in\N_{\geq 4}$ the equation $\max_{\overline{X}_n\in S(n)}|X_n|=a$ has infinitely many solutions.

\bigskip
We firmly believe that many results proved in this paper can be appropriately generalized in the context of the Diophantine equation
$$
\sigma_{i}(\overline{X}_{n})=\sigma_{n}(\overline{X}_{n}),
$$
where $i\in\{3,\ldots, n\}$ is fixed. However, even the case $i=3$ will require new ideas. For example, based on numerical calculations we formulate the following.

\begin{conj}
Let $n\in\N_{\geq 4}$ and suppose that $m\in\N$ satisfies
$$
m=\sigma_{3}(\overline{X}_{n})=\sigma_{n}(\overline{X}_{n}).
$$
Then
\begin{align*}
 m\leq \frac{1}{12} &\ \left(3021 n^6-77575 n^5+920361 n^4-6235705  n^3\right.\\
   &\ \left. \quad+24764202 n^2-53664304n+48986640\right)
\end{align*}
and
$$
x_{n}\leq \frac{1}{12} \left(27 n^4-190 n^3+471 n^2-500 n+216\right).
$$
\end{conj}

\bigskip

\noindent {\bf Acknowledgements.} We are grateful to {\L}ukasz Pa\'{n}kowski and B{\l}a\.{z}ej \.{Z}mija for interesting discussions concerning various approaches to Theorem \ref{mainbound}  .

\vskip 1cm

\noindent Piotr Miska, Jagiellonian University, Faculty of Mathematics and Computer Science, Institute of
Mathematics, {\L}ojasiewicza 6, 30-348 Krak\'ow, Poland; email:
piotr.miska@uj.edu.pl

\vskip 1cm

\noindent Maciej Ulas, Jagiellonian University, Faculty of Mathematics and Computer Science, Institute of
Mathematics, {\L}ojasiewicza 6, 30-348 Krak\'ow, Poland; email:
maciej.ulas@uj.edu.pl


\begin{thebibliography}{100}
\bibitem{Br} M. L. Brown, {\it On the diophantine equation $\sum X_{i}=\prod X_{i}$}, Math. Comput. 42 (1984), 239--240.

\bibitem{Eck} M. W. Ecker, {\it When Does a Sum of Positive Integers Equal Their Product?}, Math. Mag. 75(1) (2002), 41--47.


\bibitem{Ga} K. Gajdzica, {\it Discriminants of special quadrinomials}, Rocky Mountain J. Math. to appear.

\bibitem{Guy} R. Guy, Unsolved Problems in Number Theory. Springer-Verlag, New York, Heidelberg, Berlin, 2004. Third edition.

\bibitem{Mag} W. Bosma, J. Cannon, and C. Playoust, {\it The Magma algebra system. I. The user language}, J. Symbolic Comput., 24 (1997), 235--265.
\bibitem{Mi} M. Misiurewicz, {\it Ungel\"{o}ste Probleme}, Elem. Math. 21 (1966), 90.

\bibitem{Nyb} M. A. Nyblom, {\it Sophie Germain primes and the exceptional values of the equal-sum-and-product-problem}, Fib. Quarterly 50 (1) (2012),  58--61.

\bibitem{Sch} A. Schinzel, {\it Triples of positive integers with the same sum and the same product}, Serdica Math. J. 22 (1996), 587--588.

\bibitem{Wei} A. Weingartner, {\it On the Diophantine equation $\sum x_{i}=\prod x_{i}$}, Integers 12 (2012), Paper No. A57, 8 pp.

\end{thebibliography}
\end{document}